\documentclass[12pt]{amsart}
\usepackage{amssymb,latexsym}
\usepackage{enumerate}
\usepackage{hyperref}
\usepackage{fullpage}

\makeatletter \@namedef{subjclassname@2010}{%
  \textup{2010} Mathematics Subject Classification}
\makeatother

\newcounter{thm} \numberwithin{thm}{section}
\newtheorem{Theorem}[thm]{Theorem}

\newtheorem{Lemma}[thm]{Lemma}

\newtheorem*{Conjecture}{Conjecture}
\newtheorem*{Definition}{Definition}

\newtheorem*{UnnumberedTheorem}{Theorem}

\newcommand{\eps}[0]{\varepsilon}

\renewcommand{\mod}[1]{\ (\text{mod }#1)}

\newcommand{\lr}[1]{\left(#1\right)}

\begin{document}

\title{Long regularly-spaced and convex sequences in dense sets of integers}

\author[Brandon Hanson]{Brandon Hanson} \address{University of Georgia\\
Athens, GA}
\email{brandon.w.hanson@gmail.com}
\date{}
\maketitle
\begin{abstract}
    Let $A$ be a set of integers which is dense in a finite interval. We establish upper and lower bounds for the longest regularly-spaced and convex sequences in $A$ and in $A-A$. 
\end{abstract}
\section{Introduction}
One of the fundamental problems in additive combinatorics is that of finding structured components in a set $A$, when $A$ is a subset of an abelian group satisfying some combinatorial conditions. For instance, Szemer\'edi's celebrated theorem states:
\begin{UnnumberedTheorem}[Szemer\'edi]
Let $\delta\in(0,1]$ and let $k\geq 3$ be an integer. Then if $N$ is sufficiently large in terms of $\delta$ and $k$, any subset $A$ of $[N]$ with $|A|\geq \delta N$ contains a non-trivial $k$-term arithmetic progression.\footnote{Here and throughout the article, we use the notation $[N]=\{1,\ldots,N\}$ whenever $N$ is a positive integer.}
\end{UnnumberedTheorem}
Szemer\'edi's theorem now has several proofs, but none of them are particularly easy. Furthermore, strong quantitative dependencies on the parameters $\delta$ and $k$ are not known, and certainly the state of the art is far from what has been conjectured:
\begin{Conjecture}[Erd\H{o}s-Tur\'an]\label{ErdosTuran}
Let $A$ be an infinite sequence of integers such that
\[\sum_{a\in A}\frac{1}{a}=\infty.\]
Then $A$ contains a non-trivial $k$-term arithmetic progression for every $k\geq 1$.
\end{Conjecture}

With the goal of improving quantitative dependencies, we will investigate what happens if one relaxes the notion of arithmetic progression. To that end we make the following definition.
\begin{Definition}
Let $a_1<\ldots<a_n$ be an increasing sequence of real numbers and let $L\geq 1$ be a parameter. We say the sequence is $L$-regular if there is a positive real number $X$ such that \[X\leq a_{i+1}-a_i\leq LX\]
holds for $i=1\ldots, N-1$.
\end{Definition}

We remark that a sequence $a_1<\ldots<a_N$ forms a non-trivial arithmetic progression if
\[\frac{a_{m+1}-a_m}{a_{n+1}-a_n}=1\]
for any appropriate choice of $m,n$, while the sequence being $L$-regular is equivalent to the bi-Lipschitz condition
\[\frac{1}{L}\leq \frac{a_{m+1}-a_m}{a_{n+1}-a_n}\leq L\]
for any appropriate choice of $m,n$. 

Our first basic question is whether or not an additively structured set must contain a long $L$-regular sequence.\footnote{Since many theorems in additive combinatorics are stated in terms of a set rather than a sequence, we will repeatedly blur the distinction between a finite set of real numbers and a strictly increasing finite sequence of real numbers.} If $A$ is a set of real numbers, we define
\[R_L(A)=\max\{|A'|:A'\subseteq A,\ A'\text{ is }L\text{-regular}\}.\]
The purpose of this article is to prove estimates $R_L(A)$ under various hypotheses on $A$, a problem which can be interpreted as one rooted in Ramsey theory. However, regularly-spaced sequences are also a stepping stone to other types of structure. Indeed, results about $L$-regular sequences in dense sets were investigated previously in \cite{FraserYu}, motivated by questions of dimension in geometric measure theory. We will recall their result for context later in the article. Our principle motivation for studying regularly-spaced sequences is their connection to convex sequences. A sequence $a_1<\ldots<a_N$ is said to be (strictly) convex if the sequence of first differences $a_{i+1}-a_i$ is (strictly) monotone. The relationship between additive structure and convexity has been the subject of much investigation, and some results of various flavours linking the two can be found in \cite{RuzsaZhelezov}, \cite{SchoenShkredov}, and \cite{HRNR}. The upcoming Theorem \ref{RegularToConvex} shows that long $L$-regular sequences contain long convex subsequences. Thus, by finding long $L$-regular sequences in dense sets of integers (and other sets with similar structure), we also find rather long convex sequences in those sets. This is a result which bears some resemblance to the Erd\H{o}s-Szekeres theorem, which states that any real-sequence of length $N$ contains a monotone subsequence of length $\sqrt N$. From this perspective, Theorem \ref{RegularToConvex} can be interpreted as a strict version of the Erd\H{o}s-Szekeres theorem for first differences. A discussion of this problem is found at \cite{MathOverflow}.

We now state the main results of this paper. Estimates concerning $R_L(A)$ are most interesting when $L$ is reasonably small. We will focus on the case $L=2$, but the following lemma (proved in Section \ref{RegularityAndCoveringSection}) shows that case of smaller values of $L$ can be reduced to $L=2$ at a modest quantitative expense.
\begin{Lemma}\label{SmallerL}
If $A$ is $2$-regular and $l\geq 2$ is an integer, then $A$ has a $\lr{1+\frac{1}{l}}$-regular subset $A'$ with
\[|A'|\geq \frac{|A|}{4l+2}.\]
\end{Lemma}
The first result concerning $R_2(A)$ is not original, but one which can be extracted from \cite{FraserYu} and serves as a reference point for our estimates. It is a quantitatively strong $L$-regular analogue of Szemer\'edi's theorem. 
\begin{UnnumberedTheorem}[Fraser-Yu]\label{FY}
Let $\eps>0$ be a fixed constant. Let $A\subseteq[N]$ be a set of size $\delta N$. Then \[R_2(A)\gg_\eps \frac{(\log N)^{1-\eps}}{\log(1/\delta)}.\]
\end{UnnumberedTheorem}

Under slightly different hypotheses, one can locate regularly-spaced sequences which are much longer than those provided by the Fraser-Yu theorem. Our first theorem, below, shows that there is a significant quantitative improvement to be had by considering colourings instead of density. It can be thought of as an $L$-regular analogue of van der Waerden's theorem.

\begin{Theorem}\label{Colouring}
Let $r$ and $N$ be positive integers such that  $N\geq 3^{r^2+r}$. Then if $[N]=A_1\cup\cdots\cup A_r$ is a partition of $[N]$, we have
\begin{equation}\label{Colouring1}
    \max\{R_2(A_i):1\leq i\leq r\}\geq \frac{N^{1/r}}{3^r}.
\end{equation}
\end{Theorem}

In \cite{Bourgain}, it is proved that Szemer\'edi-type theorems are quantitatively improved by looking at sumsets of dense sets, rather than dense sets themselves. Further results in this direction can be found in, for instance, \cite{Green},
\cite{CrootRuzsaSchoen}, \cite{CrootSisask}, \cite{CrootLabaSisask}, and \cite{FHR}. The following theorem shows that very long regular sequences can be found in difference sets as well. For very dense sets we have the following.
\begin{Theorem}\label{DenseDifference}
Let $N$ be a positive integer and $\delta\in(0,1]$ be such that $N\geq 3^{\lfloor2/\delta\rfloor^2+\lfloor2/\delta\rfloor}$. Suppose $A\subseteq [N]$ has density at least $\delta$. Then \begin{equation}\label{DenseDifference1}
    R_2(A-A)\geq \frac{N^{\delta/2}}{3^{2/\delta}}.
\end{equation}
\end{Theorem}

For sparser sets, we can apply a density increment strategy to improve on Theorem \ref{DenseDifference}. The result is a much milder dependence on $\delta$.

\begin{Theorem}\label{SparseDifference}
Let $\delta\in(0,1]$ and set $s=\lceil \log_2(1/4\delta)\rceil$. Suppose $N\geq\max\{(2s)^s,3^{73s}\}$ is a positive integer and that $A\subseteq [N]$ has density at least $\delta$. Then
\begin{equation}\label{SparseDifference1}
    R_2(A-A)\geq 3^{-8}N^{1/(8s+8)}.
\end{equation}
\end{Theorem}

We next relate the notions of regularity and convexity. This allows us to deduce theorems similar to those preceding but for convex sequences. We emphasize the consequences of Theorem \ref{Colouring} and Theorem \ref{SparseDifference}.
\begin{Theorem}\label{RegularToConvex}
Let $A$ be a set of integers and let $C(A)$ denote the length of the longest strictly convex sequence contained in $A$. Then
\[C(A)\geq\left\lfloor\frac{R_2(A)^{1/2}}{4}\right\rfloor.\]
Consequently, if $[N]=A_1\cup\cdots\cup A_r$ is a partition of $[N]$ and $N\geq 3^{r^2+r}$, then
\begin{equation}\label{Colouring2}
    \max\{C(A_i):1\leq i\leq r\}\geq \left\lfloor\frac{N^{1/2r}}{4\cdot 3^{r/2}}\right\rfloor.
\end{equation}
Similarly, if $A\subseteq[N]$ has density at least $\delta$ and $N\geq\max\{(2s)^s,3^{73s}\}$ then \begin{equation}\label{ConvexDifference1}
    C(A-A)\geq \left\lfloor\frac{N^{1/(16s+16)}}{4\cdot3^8}\right\rfloor,
\end{equation}
where $s=\lceil \log_2(1/4\delta)\rceil$.
\end{Theorem}

To complement the theorems above, we provide a number of constructions, the first of which demonstrates that estimates (\ref{Colouring1}) and (\ref{Colouring2}) are pretty sharp. 
\begin{Theorem}\label{ColouringConstruction}
For any positive integer $r$, there is an arbitrarily large integer $N$ such that $[N]$ has a partition $[N]=A_1\cup\cdots\cup A_r$ with the properties
\[\max\{R_2(A_i):1\leq i\leq r\}\leq 2(r-1)!N^{1/r}\]
and
\[\max\{C(A_i):1\leq i\leq r\}\leq r!N^{1/r}.\]
\end{Theorem}

We also construct dense sets $A\subseteq [N]$ for which $R_2(A)$ and $C(A)$ are smaller than a power of $\log N$. This construction shows that the structure of a difference set is crucial to improving the estimates (\ref{DenseDifference1}), (\ref{SparseDifference1}) and (\ref{ConvexDifference1}).
\begin{Theorem}\label{DensityConstruction}
For arbitrarily large $N$, there is a set $A\subseteq [N]$ with $|A|\geq N/2$ such that \[R_2(A)\leq 16\lr{\frac{\log N}{\log\log N}}^2\]
and
\[C(A)\leq  24\lr{\frac{\log N}{\log\log N}}^3.\]
\end{Theorem}

On the other hand, we can also construct difference sets $A-A$ containing no long 2-regular or convex subsets, which necessitates some sort of density hypothesis in Theorems \ref{DenseDifference}, \ref{SparseDifference} and \ref{RegularToConvex}. 
\begin{Theorem}\label{DifferenceConstruction}
For arbitrarily large $n$, there is a set $A\subset[2\cdot 16^n]$ with $|A|=2^n$ and such that \[R_2(A-A)\leq 3\]
and
\[C(A-A)\leq 2n.\]
\end{Theorem}

The rest of this paper structured is as follows. In Section \ref{RegularityAndCoveringSection}, we develop some preliminary ideas which will be used in the proofs of the main theorems and then present the proof of Lemma \ref{SmallerL} as well as a proof of Theorem \ref{RegularToConvex} which is conditional on Theorems \ref{Colouring} and \ref{SparseDifference}. In Section \ref{ColouringSection} we prove Theorem \ref{Colouring}, and in Section \ref{DifferenceSetsSection} we prove Theorems \ref{DenseDifference} and \ref{SparseDifference}. Section \ref{ConstructionsSection} contains the constructions which prove Theorems \ref{ColouringConstruction}, \ref{DensityConstruction} and \ref{DifferenceConstruction}.

\section*{Acknowledgements}
The author has benefited from insightful discussions with Giorgis Petridis. He is also supported by the NSF Award 2001622.

\section{Regularity and Covering}\label{RegularityAndCoveringSection}
In this section we will discuss some basic tools for dealing with regular sequences that, while rather simple, are fundamental in this article. Already, these tools enable us to prove Lemma \ref{SmallerL} and (conditionally) Theorem \ref{RegularToConvex}. We begin with a definition.
\begin{Definition}
By consecutive intervals, we mean a sequence of intervals of the form $I_1,\ldots,I_k$ with \[I_j=[s+jl,s+(j+1)l)\] for some real numbers $s$ and $l$. Suppose $A=\{a_1<\ldots<a_N\}$ is a sequence of real numbers. If $M$ is a positive integer, we say $A$ is $M$-covered if there are consecutive intervals $I_1,\ldots,I_k$ such that: \begin{enumerate}
\item for each $j$ we have $1\leq |A\cap I_j|\leq M$, and
\item we have $A\subseteq\bigcup_{j=1}^k I_j$. 
\end{enumerate}
\end{Definition}
The relevance of $M$-covering is illustrated by the following lemma.
\begin{Lemma}\label{Covering}
Let $A$ be a set of real numbers which admits an $M$-covering for some $M\leq |A|$. Then 
\[R_2(A)\geq \frac{|A|}{3M}.\]
Conversely, if $A$ is an $L$-regular sequence, then $A$ can be $\lceil L\rceil$-covered.
\end{Lemma}
\begin{proof}
Let $I_1,\ldots,I_k$ be the intervals of the $M$-covering of $A$, and let $l$ be the length of each $I_j$. Because each interval contains at most $M$ elements of $A$, we must have $k\geq |A|/M$. Choose an element $a_j$ of $A$ from each interval $I_j$ with $j=1\mod 3$, and consider the resulting sequence. Because
\[2l\leq a_{j+3}-a_j\leq 4l,\]
the sequence is 2-regular, and the number of $a_j$ selected is at least $|A|/3M$. 

To prove the converse, we write $A=\{a_1<\ldots<a_n\}$ with
\[X\leq a_{i+1}-a_i\leq LX\] for $i=1,\ldots,n-1$. Then $A$ can be covered by consecutive intervals of length $LX$ such that each contains at least one element of $A$ and none of them contains more than $\lceil L\rceil$ elements of $A$. 
\end{proof}

\begin{proof}[Proof of Lemma \ref{SmallerL}]
By definition, there is a number $X$ so that
\[X\leq a_{i+1}-a_i\leq 2X\] for each $i$. Let $q=2l+1$. Now cover the set $A$ by consecutive intervals $[s, s+2X)$ of length $2X$. Each interval can contain at most two elements of $A$, so there are at least $|A|/2$ consecutive intervals which intersect $A$. Choose from every $q$'th such interval an element of $A$, and let $A'=\{a_1'<\ldots<a_n'\}$ be the resulting set. Then there are $q-1$ empty intervals between consecutive elements of $A'$, while two consecutive elements are contained in $q+1$ such intervals, so we have \[(q-1)2X\leq a_{i+1}'-a_i'
\leq (q+1)2X.\]
Since \[(q+1)/(q-1)=1+\frac{1}{l},\]
the set $A'$ is $(1+1/l)$-regular.
Finally $A'$ has size
\[|A'|\geq\frac{|A|}{2q}.\]
\end{proof}

At this point, we have enough to prove the first claim of Theorem \ref{RegularToConvex}. The other claims will follow once we have established Theorems \ref{Colouring}, \ref{DenseDifference} and \ref{SparseDifference} in Sections \ref{ColouringSection} and \ref{DifferenceSetsSection}.

\begin{proof}[Proof of Theorem \ref{RegularToConvex}, first claim]
By Lemma \ref{Covering}, a $2$-regular sequence $B$ can be 2-covered by consecutive intervals of length $2X$, say $I_j=[s+(j-1)2X,s+j2X)$. Now pick $s_{i_j}$ to be any term from the sequence belonging to $I_{2j^2}$. Then
\[2X(4j)\leq s_{i_{j+1}}-s_{i_j}\leq 2X(4j+2)\]
since there are there are $4j$ intervals in between $I_{2j^2}$ and $I_{2(j+1)^2}$. From this one sees that the gaps between $s_{i_{j+1}}$ and $s_{i_j}$ are increasing as desired. Since the whole sequence intersects $|B|/2$ intervals, the convex subsequence has length at least $\lfloor|B|^{1/2}/4\rfloor$.
\end{proof}

\section{Long monochromatic regular sequences}\label{ColouringSection}
Here we prove Theorem \ref{Colouring}, which is the foundation for all subsequent results.
\begin{proof}[Proof of Theorem \ref{Colouring}]
We proceed by induction on $r$, and when $r=1$ there is nothing to do. Suppose now we know the theorem holds with $r-1$ parts, and assume that of these parts, $A_r$ has the largest cardinality so that $|A_r|\geq N/r$. Let $M$ denote the maximum difference between consecutive elements of $A_r$. Then $A_r$ is $M$-covered by intervals of length $M$, so by Lemma \ref{Covering} there is a 2-regular sequence in $A_r$ of length at least $N/(3Mr)$. This is already sufficient if $M$ is such that \[\frac{N}{3Mr}\geq \frac{N^{1/r}}{3^r}.\] Otherwise
\[M\geq \frac{3^rN^{(r-1)/r}}{3r}\] and, by the definition of $M$, there is a subinterval $I$ of $[N]$ of length $M-1$ which is disjoint from $A_r$. The interval 
$I$ is partitioned as $(I\cap A_1)\cup\cdots\cup (I\cap A_{r-1})$. Because the problem is translation invariant, we may appeal to induction and we get 
\[\max\{R_2(A_i\cap I):1\leq i\leq r-1\}\geq \frac{|I|^{1/(r-1)}}{3^{r-1}},\]
provided $|I|^{1/(r-1)}\geq 3^{r}$. We will show the stronger estimate \[|I|^{1/(r-1)}\geq \frac{N^{1/r}}{3}\]
which will also close the induction. Indeed,
\begin{align*}
    |I|^{1/(r-1)}&=(M-1)^{1/(r-1)}\\
    &\geq \lr{\frac{3^rN^{(r-1)/r}}{3r}-1}^{1/(r-1)}\\
    &=\frac{N^{1/r}}{3}\lr{\frac{3^{2r-2}}{r}-\frac{3^{r-1}}{N^{(r-1)/r}}}^{1/(r-1)}\\
    &\geq \frac{N^{1/r}}{3}\lr{\frac{3^{2r-2}}{r}-\frac{1}{3}}^{1/(r-1)}\\
    &\geq \frac{N^{1/r}}{3}.
\end{align*}
\end{proof}

\section{Long regular sequences in difference sets}\label{DifferenceSetsSection}

\begin{Lemma}[Ruzsa's Covering Lemma]
If $A$ and $B$ are finite subsets of an abelian group then there is a set $X\subseteq B$ with $|X|\leq \frac{|A+B|}{|A|}$ and such that $B\subseteq A-A+X$.
\end{Lemma}
We include the proof for the sake of completeness.
\begin{proof}
Let $X$ be a maximal (with respect to inclusion) subset of $B$ such that the translates $A+x$ are disjoint for $x\in X$. Then, for any $b\in B$,
\[(b+A)\cap (x+A)\neq \varnothing.\]
From this, there exists $a,a'\in A$ such that $b+a=x+a'$ and the containment $B\subseteq A-A+X$ follows. Further, 
\[|X||A|=|A+X|\leq |A+B|\leq \frac{|A+B|}{|A|}|A|,\]
which gives the desired estimate on $|X|$.
\end{proof}

\begin{proof}[Proof of Theorem \ref{DenseDifference}]
By Ruzsa's Covering Lemma, there is a set $X\subseteq[N]$ with 
\[|X|\leq \frac{2N}{|A|}\leq \frac{2}{\delta},\]
and such that
\[[N]\subseteq A-A+X.\]
Let $r=\lfloor 2/\delta\rfloor$, so that $r\geq |X|$. We define a colouring of $[N]$ with at most $r$ parts by colouring $n$ with $x\in X$ if $x$ is minimal subject to the constraint $n\in A-A+x$. Apply Theorem \ref{Colouring}, from which we find an $x$ such that
$A-A+x$ contains a 2-regular sequence of length at least $N^{1/r}3^{-r}$. 
\end{proof}

\begin{proof}[Proof of Theorem \ref{SparseDifference}]
Let \[k=\left\lfloor N^{1/s}\right\rfloor.\] Then $k^s\leq N\leq 2k^{s}$ by the assumed size of $N$. It follows there is some subinterval $I$ of $[N]$, with length $k^s$ and in which $A$ has density at least $\delta/2$. Moreover, the problem is translation invariant. Thus by replacing $[N]$ with $I$, replacing $\delta$ with $\delta/2$ (which has the effect of replacing $s$ with $s-1$), and replacing $A$ with $A\cap I$, it suffices to show the following.

\textit{Claim}: If $N=k^s$ and $A\subseteq [N]$ has density at least $\delta$ then $R_2(A-A)\geq 3^{-8}N^{1/8s}$. 

Let $M=k^{s-1}$ and for each $n\in[N]$ write
\[n=q_nM+r_n\] where $1\leq r_n\leq M$ and $0\leq q_n<k$. For $q\leq k$, let \[A_2(q)=\{r\in \{1,\ldots, M\}:qM+r\in A\},\] and let $A_1$ denote the set of those $q$ for which $A_2(q)$ is non-empty. 

Suppose first that $A_1+1$ has density at least $1/2$ in $[k]$. Then by Theorem \ref{DenseDifference}, $A_1-A_1$ contains a 2-regular sequence of length $l$ with $l\geq k^{1/4}/3^4$, noting that $k\geq 3^{20}$. Let $q_1-q_1'<\ldots<q_l-q_l'$ denote this sequence. Because the $q_i-q_i'$ are integers, there is an integer $X\geq 1$ so that \[X\leq (q_{i+1}-q_{i+1}')-(q_{i}-q_{i}')\leq 2X.\]
For each $i=1,\ldots,l$, there are integers $r_i,r_i'\leq M$ and $a_i,a_i'\in A$ so that
\[a_i-a_i'=(q_i-q_i')M+r_i-r_i'.\]
Since $-M\leq r_i-r_i'\leq M$,
\[XM-2M\leq (a_{i+1}-a_{i+1}')-(a_{i}-a_{i}')\leq 2XM+2M.\]
From this we see that the set $D=\{a_i-a_i':i=1,\ldots,s\}\subseteq A-A$ is $L$-regular with
\[L\leq \frac{2MX+2M}{MX-2M}\leq \frac{2MX-4M+6M}{MX-2M}=2+\frac{6}{X-2}\leq 8,\]
provided $X\geq 3$. In this case we can $8$-cover $D$, so that \[R_2(A-A)\geq R_2(D)\geq  \frac{l}{24}\] by Lemma \ref{Covering}. If $X\leq 2$, then the inequalities \[(a_{i+1}-a_{i+1}')-(a_i-a_i')=((q_{i+1}-q_{i+1}')-(q_i'-q_i'))M+(r_{i+1}-r_{i+1}')-(r_i-r_{i}')<6M\]
and \[(a_{i+8}-a_{i+8}')-(a_i-a_i')=((q_{i+8}-q_{i+8}')-(q_i'-q_i'))M+(r_{i+8}-r_{i+8}')-(r_i-r_{i}')>6M\] show that the set $D$ can be $8$-covered by intervals of length $6M$, and so
\[R_2(A-A)\geq R_2(D)\geq \frac{l}{24}\] by Lemma \ref{Covering}. In either case, we have proved $R_2(A-A)\geq k^{1/4}/(3^4\cdot 24)$ which is good enough.

Now suppose $|A_1|\leq k/2$. Then, because
\[|A|=\sum_{q\in A_1}|A_2(q)|\leq |A_1|\max_q |A_2(q)|,\] there is some $q$ with $A_2(q)\geq 2\delta M$, and we note the fact that $A_2(q)-A_2(q)\subseteq A-A$. We iterate this argument, replacing $A$ with $A_2(q)$, $[N]$ with $[M]$, and $\delta$ with $2\delta$ at each step, but stopping if $A_1+1$ has density at least $1/2$ in $[k]$, in which case we are done by the argument above. After $s-1$ iterations we have a set $A'\subseteq[k]$ which has density at least $2^{s-1}\delta$ and with the property $A'-A'\subseteq A-A$. By definition, $s\geq \log_2(1/2\delta)$ so $2^{s-1}\delta\geq 1/4$ and by Theorem \ref{DenseDifference} (this time noting that $k\geq 3^{72}$) we get
\[R_2(A'-A')\geq \frac{k^{1/8}}{3^8}.\]
\end{proof}

\section{Constructions}\label{ConstructionsSection}
In this section we provide various constructions which prove Theorems \ref{ColouringConstruction}, \ref{DensityConstruction} and \ref{DifferenceConstruction}.
\begin{proof}[Proof of Theorem \ref{ColouringConstruction}]
We will assume $N=M^r$ for $M$ arbitrarily large. The partition is most easily constructed if one thinks of it as the fibers of a colouring $c_r:[N]\to \{0,\ldots,r-1\}$. In this perspective, we want to construct a function $c_r$ such that $c_r$ is not constant on any sufficiently long $2$-regular sequence. The colourings $c_r$ are constructed by induction on $r$. When $r=1$, the statement is trivial. For larger $r$, we will colour $[N]$ with the residue classes modulo $r$. To begin, divide $[N]$ into $M$ intervals \[I_k=\{kM^{r-1}+1,\ldots,(k+1)M^{r-1}\},\] of length 
$M^{r-1}$. By induction, we can $(r-1)$-colour each $I_k$ using the set of colours \[C_{k\mod r}=\{j:j\in\{1,\ldots, r\},\  j\not\equiv k\mod r\}\] 
in such a way that $I_k$ contains neither a monochromatic sequence of length $2(r-2)!M$ which is $2$-regular, nor a monochromatic sequence of length $(r-1)!M$ which is strictly convex. In this way, we have $r$-coloured all of $[N]$. 

Now let $A$ be a monochromatic sequence in $[N]$ which is $2$-regular. If $A$ intersects at most $r-1$ of the intervals $I_k$, then it must intersect a single $I_k$ in at least $|A|/(r-1)$ consecutive elements by the pigeonhole principle. Since any subsequence of consecutive elements of $A$ is also 2-regular, we conclude from induction that
\[\frac{|A|}{r-1}\leq |A\cap I_k|\leq 2(r-2)!M,\]
giving the desired bound on $|A|$. The other possibility is that $A$ intersects at least $r$ different intervals $I_k$. Suppose the colour class of $A$ is $j$. By construction, given any $r$ consecutive intervals $I_k$, one of the intervals avoids the $j$'th colour class completely. Thus some gap between consecutive elements of $A$ must be large enough to avoid such an interval, that is, there are two consecutive elements of $A$ which are separated by at least $M^{r-1}$. It follows that any two consecutive elements of $A$ are separated by at least $M^{r-1}/2$ and so 
\[|A|\leq 2M\] which closes the induction.

Next we handle the convex scenario. To analyse this case, it helps to write \[A=\{a_1<\ldots<a_n\}.\]
Suppose there are two consecutive elements, say $a_{i}$ and $a_{i+1}$, with $a_{i+1}-a_i\geq M^{r-1}$ and let assume $i$ is minimal. Then, $a_{j+1}-a_j\geq M^{r-1}$ for $j\geq i$ by convexity, so $n-i\leq M$. The remaining part of $A$ is the subsequence $A'=\{a_1<\ldots<a_i\}$, and it has consecutive differences strictly less than $M^{r-1}$. Reasoning as in the 2-regular case, $A'$ is contained in some $r-1$ consecutive intervals. One of these intervals contains at least $i/(r-1)$ consecutive points from $A'$, and these points form a strictly convex sequence. Again by induction
\[\frac{i}{r-1}\leq (r-1)!M\]
so we arrive at the estimate
\[|A|=i+(n-i)\leq (r-1)(r-1)!M+M\leq r!M.\]
\end{proof}

Theorem \ref{DensityConstruction} requires the following lemma, in which we construct a Cantor-like set off of which our construction is based. 

\begin{Lemma}\label{Cantor}
Let $k,K\geq 2$ be integer parameters and let $N=(K-1)(2K)^k$. Then there is a decreasing sequence of sets \[A_k\subseteq A_{k-1}\subseteq \ldots\subseteq A_0=[N]\] such that
\begin{enumerate}
\item $|A_{i}|\geq (1-1/K)|A_{i-1}|$,
\item for $i\geq 1$, $A_i$ is a union of intervals $I_{i,j}$ of length $(K-1)N_i$, where $N_i=(2K)^{k-i}$, 
\item any 2-regular subset of $A_i$ which is larger than $4K^2-4K$ must be confined to a single interval $I_{i,j}$, and
\item any convex subset of $A_i$ has all but at most $i(2K^2-2K)$ of its elements in a single interval $I_{i,j}$.
\end{enumerate}
\end{Lemma}
\begin{proof}
To construct $A_{i}$ from $A_{i-1}$, we simply divide each interval $I_{i-1,j}$ into intervals of length $N_{i}$ and remove from $A_{i-1}$ every $K$'th such interval. To be precise, at stage $i-1$, we have a set $A_{i-1}$ which is a union of intervals $I_{i-1,j}$ of length $(K-1)N_{i-1}$. Let \[I_{i-1,j}=[a_{i-1,j},a_{i-1,j}+(K-1)N_{i-1})\] be such an interval. We divide $I_{i-1,j}$ into intervals of the form
\[[a_{i-1,j}+lN_{i},a_{i-1,j}+(l+1)N_{i}).\]
Then, we remove from $I_{i-1,j}$ (and thus $A_{i-1}$) those intervals with $K$ dividing $l$. Note that $N_{i-1}=2KN_{i}$, and so $I_{i-1,j}$ can be evenly divided into $2(K-1)K$ intervals of length $N_i$.

Now we verify that properties (1), (2), (3) and (4) hold. By construction, at stage $i$ we only remove a proportion of $1/K$ elements from $A_{i-1}$. This shows $|A_{i}|\geq (1-1/K)|A_{i-1}|$ which proves property (1).

The fact that $K-1$ consecutive intervals out of every $K$ intervals in $I_{i-1,j}$ are not removed from $A_{i-1}$ shows that property (2) holds. 

For property (3) we work inductively. Suppose $A'$ is a 2-regular sequence in $A_{i}$ of size greater than $4K^2-4K$. Then $A'$ is necessarily confined to a single interval $I_{i-1,j}$ by property (3) applied at the stage $i-1$. Thus $A'$ is contained in an interval of length $(K-1)N_{i-1}$ by (2). Furthermore, if $A'$ is not confined to a single interval $I_{i,j'}$, then there necessarily is a gap between consecutive elements of length at least $N_{i}$. This is due to the fact that the intervals $I_{i,j'}$ are separated by intervals of length $N_{i}$. It follows that any gap between consecutive elements of $A'$ lies between $\frac12N_{i}$ and $2N_{i}$.
Thus \[|A'|\leq \frac{(K-1)N_{i-1}}{\frac12N_{i}}=4K^2-4K.\] 

Property (4) is proved similarly. Suppose $A'$ is a convex sequence in $A_{i}$ of size $s$. All but at most $(i-1)(2K^2-2K)$ elements lie outside a single $I_{i-1,j}$, so it suffices to show at most $2K^2-2K$ of those elements of $A'$ in $I_{i-1,j}$ lie outside a single $I_{i,j'}$. However, once $A'$ exists $I_{i,j'}$, all subsequent gaps must be at least $N_i$, and are confined to an interval of length $(K-1)N_{i-1}$. So there are at most $2K^2-2K$ such elements.
\end{proof}

\begin{proof}[Proof of Theorem \ref{DensityConstruction}]
Let $k$ be sufficiently large and let $K=2k$. In this way \[N=(2k-1)(4k)^k.\] Let $A=A_k$ be as in Lemma \ref{Cantor}. By Bernoulli's inequality,
\[\lr{1-\frac{1}{K}}^k\geq 1-\frac{k}{K}=\frac{1}{2}\]
so that $A$ has density at least $1/2$.

By property (3) of Lemma \ref{Cantor}, any $2$-regular subset of $A$ has size at most \[4K^2-4K\leq 4K^2=16k^2,\] and any convex subset has size at most
\[K-1+\frac{K}{2}(2K^2-2K)\leq 3K^3=24k^3.\]
Finally, since
\[k\leq \frac{\log N}{\log\log N},\]
our estimates follow.
\end{proof}

\begin{proof}[Proof of Theorem \ref{DifferenceConstruction}]
Let $A_1=\{1\}$ and inductively let 
\[A_{i+1}=(A_i)\cup (16^{i}+A_i).\]
It is a straightforward induction to show that $A_i\subseteq [1,2\cdot 16^{i-1}]$, from which $|A_i|=2^i$ and  $A_i-A_i\subseteq (-2\cdot 16^{i-1},2\cdot 16^{i-1})$ follow. We will prove by induction that $R_2(A_i-A_i)\leq 3$ and $C(A_i-A_i)\leq 2i$. The case $i=1$ is trivial. Now
\[A_{i+1}-A_{i+1}=(-16^i+A_i-A_i)\cup(A_i-A_i)\cup(16^i+A_i-A_i).\]
If a $2$-regular or convex sequence intersects just one of the three sets on the right then we can apply translation invariance and induction. Otherwise, any $2$-regular or convex sequence which intersects two of the three sets on the right contains a gap of length at least $16^i/2$. In the case of a regular sequence, the minimum gap must then be at least $16^i/4$, and so the sequence has at most three terms (one from each set). In the case of a convex sequence, there can only be multiple elements in one of the three sets. So by induction we have $C(A_{i+1}-A_{i+1})\leq C(A_i-A_i)+2\leq 2(i+1)$.
\end{proof}

\end{document}